\documentclass[11pt]{article}
\usepackage{amsmath,amsthm,amssymb,amsfonts}
\usepackage{enumerate,amscd,amsxtra}
\usepackage{mathrsfs}
\usepackage{color}
\usepackage[cmtip,all]{xy}

 \normalsize

\newtheoremstyle{theoremstyle}
  {10pt}      
  {5pt}       
  {\itshape}  
  {}          
  {\bfseries} 
  {:}         
  {.5em}      
  {}          

\newtheoremstyle{examplestyle}
  {10pt}      
  {5pt}       
  {}          
  {}          
  {\bfseries} 
  {:}         
  {.5em}      
  {}          

\theoremstyle{theoremstyle}
\newtheorem{theorem}{Theorem}[section]
\newtheorem*{theorem*}{Theorem}
\newtheorem{lemma}[theorem]{Lemma}
\newtheorem{proposition}[theorem]{Proposition}
\newtheorem*{proposition*}{Proposition}
\newtheorem{corollary}[theorem]{Corollary}
\newtheorem*{corollary*}{Corollary}

\newtheorem{example}[theorem]{Example}
\newtheorem{definition}[theorem]{Definition}
\newtheorem{definition*}{Definition}
\newtheorem{remark}[theorem]{Remark}
\newtheorem{remark*}{Remark}

\newcommand{\caC}{{\mathcal C}}
\newcommand{\caD}{{\mathcal D}}

\newcommand{\caS}{{\mathcal S}}
\newcommand{\caO}{{\mathcal O}}

\newcommand{\integers}{{\mathbb Z}}

\newcommand{\Ho}{\mathsf{Ho}}

\newcommand{\Hom}{\mathrm{Hom}}

\newcommand{\Spt}{\mathsf{Spt}}

\newcommand{\unit}{\mathbf{1}}

\newcommand{\Mod}{{\mathrm{Mod}}}

\newcommand{\J}{\mathcal{J}}
\newcommand{\m}{\underline{\mathrm{m}}}
\newcommand{\n}{\underline{\mathrm{n}}}
\renewcommand{\k}{\underline{\mathrm{k}}}
\renewcommand{\l}{\underline{\mathrm{l}}}
\renewcommand{\i}{\underline{\mathrm{i}}}

\newcommand{\p}{\underline{\mathrm{p}}}

\newcommand{\Set}{\mathsf{Set}}
\newcommand{\sSet}{\mathsf{sSet}}
\newcommand{\E}{\mathsf{E}}
\newcommand{\sS}{\mathscr{S}}
\newcommand{\mot}{\mathrm{mot}}
\newcommand{\cC}{\mathscr{C}}
\newcommand{\cD}{\mathscr{D}}
\newcommand{\Map}{\mathrm{Map}}

\title{{\bf Another viewpoint on $\J$-spaces}}
\author{Markus Spitzweck}
\date{\today}
\begin{document}
\maketitle
\begin{abstract}
We give an interpretation of $\J$-spaces in terms of symmetric spectra
in symmetric sequences. As application we show how one can define
graded endomorphism objects in a general situation. As example we discuss
the motivic bigraded endomorphisms of a motivic $E_\infty$-ring spectrum.
Finally we give an $\infty$-categorical interpretation of our result.
\end{abstract}

\section{Introduction}

The category of $\J$-spaces has been introduced in \cite{sagave-schlichtkrull}
to figure as a suitable target for the graded units of an $E_\infty$-ring
spectrum. Recall the units $\mathrm{GL}_1 (\E)$ of a fibrant $\E_\infty$-ring spectrum $\E$ are defined
to be $\Omega^\infty(\E)^\times$, i.e. those connected components of
$\Omega^\infty(\E)$ which are invertible components for the $E_\infty$-structure
on $\Omega^\infty(\E)$ coming from that one on $\E$.

Equivalently one can view $\mathrm{GL}_1 (\E)$ as the space of $\E$-module
automorphisms from $\E$ to itself. This concept does not capture
the possibility of graded automorphisms of $\E$, i.e. self-equivalences
from a (positive or negative) suspension of $\E$ to $\E$.

In particular the canonical map $\mathrm{GL}_1(\mathsf{e}) \to \mathrm{GL}_1 (\E)$
is always an equivalence, where $\mathsf{e} \to \E$ is the connected cover.

The need for some graded version of $\mathrm{GL}_1 (\E)$ stems in particular  
from the theory of topological logarithmic structures, see \cite{rognes.logarithmic}.

In \cite{sagave-schlichtkrull} the graded units of an $E_\infty$-ring spectrum
manifest themselves as {\em $\J$-spaces}. In loc.~cit.~it is discussed that
only in the context of {\em graded} topological logarithmic structures
there is an interesting such structure on the connective cover of the topological
complex $K$-theory spectrum coming from the Bott element.

In this note we discuss $\J$-objects in general (symmetric monoidal) categories
and model categories. Our main result (Proposition \ref{bgdedd}) says that the category of $\J$-objects
is equivalent to that one of symmetric $T$-spectra, where $T$ is a particular
symmetric sequence.
In the proof we introduce as intermediate step a structure which we call $T$-data.
The author was informed that these data were in fact the
first manifestation of $\J$spaces Sagave-Schlichtkrull used.

Our applications are mainly for motivic $\E_\infty$-spectra,
we show how to obtain bigraded versions of the motivic endomorphisms of such a spectrum.

We also discuss the $\infty$-categorical content of our result.

\vskip.3cm

{\bf Acknowledgements:} This note grew out of a talk given by Christian Schlichtkrull
at the National Topology Symposium in Fredrikstad 2010. I thank Steffen Sagave and
Christian Schlichtkrull for interesting correspondence on the subject.
Also I thank Peter Arndt for useful discussions.

\section{$\J$-spaces as symmetric spectra}
\label{hgrfddd}

For any natural number $n$ we set $\n=\{1,\ldots,n\}$.
As in \cite{sagave-schlichtkrull} we denote by $\J$ the following category:
objects are pairs $(\m,\n)$, morphisms from $(\m,\n)$ to $(\k,\l)$
are triples $(\varphi,\psi,\alpha)$ where $\varphi \colon \m \to \k$
and $\psi \colon \n \to \l$ are injections and $\alpha \colon
(\k - \varphi(\m)) \to (\l - \psi(\n))$ is a bijection. 
Composition takes as bijection the disjoint union of the induced bijections.
$\J$ is symmetric monoidal, where on objects the tensor product is
concatenation. The symmetry isomorphism involves the shuffle permutation.

Let $\caC$ be a symmetric monoidal category with all colimits.
We equip the category of symmetric sequences $\caC^\Sigma$
and the category of $\J$-objects $\caC^\J$
with the Day convolution tensor product.
We let $T=T_\caC=(\emptyset,\unit,\emptyset,\emptyset,\ldots)$
be the symmetric sequence in $\caC$ where the tensor unit sits in degree one,
and in every other degree the initial object.
We let $\Spt_T^\Sigma(\caC^\Sigma)$ be the symmetric monoidal category
of symmetric $T$-spectra in $\caC^\Sigma$. By definition this is the category
of modules over the commutative monoid $\mathrm{Sym}(T)=(\unit,T,T^{\otimes 2}, T^{\otimes 3}, \ldots)$
in $(\caC^\Sigma)^\Sigma$.

\begin{proposition} \label{bgdedd}
The categories $\caC^\J$ and $\Spt_T^\Sigma(\caC^\Sigma)$ are
naturally equivalent as symmetric monoidal categories.
\end{proposition}

\begin{proof}
An object in $\Spt_T^\Sigma(\caC^\Sigma)$ is a symmetric sequence
$(X_0,X_1,\ldots)$ in $\caC^\Sigma$ together with bonding maps
$X_i \otimes T \to X_{i+1}$ such that the iterates
$$X_i \otimes T^{\otimes p} \to X_{i+p}$$
are $\Sigma_i \times \Sigma_p$-equivariant.

The functor $X \mapsto X \otimes T$ on $\caC^\Sigma$ has a right adjoint, denoted
$X \mapsto X^T$. We have $(X^T)_n= X_{n+1}$ with the $\Sigma_n$-action
induced by the $\Sigma_{n+1}$-action on $X_{n+1}$ via the natural
inclusion $\Sigma_n \to \Sigma_{n+1}$.

Thus the bonding map $X_i \otimes T \to X_{i+1}$ of a $T$-spectrum
$X$ is adjoint to a map $X_i \to X_{i+1}^T$. Such a map is given
by a family of maps $\varphi_{i,n} \colon X_{i,n} \to X_{i+1,n+1}$, such that the $n$-th map is
$\Sigma_n$-equivariant. Moreover, the equivariance of the iterated bonding
maps translates to the following statement: the composition
$$\Phi_{i,n,p}:= \varphi_{i+p-1,n+p-1} \circ \cdots \circ \varphi_{i,n} \colon X_{i,n} \to X_{i+p,n+p}$$
is $\Sigma_i$-equivariant and for each $g \in \Sigma_p$ we have $g \circ \Phi_{i,n,p}=\Phi_{i,n,p}$.
Here $g$ acts on $X_{i+p,n+p}$ via the inclusion
$\iota_{i,n,p} \colon \Sigma_p \to \Sigma_{i+p} \times \Sigma_{n+p}$
which is the product of the maps from $\Sigma_p$ into $\Sigma_{i+p}$ and $\Sigma_{n+p}$
which permute the last $p$ elements.

So a symmetric $T$-spectrum amounts to the following data:

objects $X_{i,n} \in \caC$ with a $\Sigma_i \times \Sigma_n$-action,
$\Sigma_i \times \Sigma_n$-equivariant maps $\varphi_{i,n} \colon X_{i,n} \to X_{i+1,n+1}$ such that
their iterates $\Phi_{i,n,p} \colon X_{i,n} \to X_{i+p,n+p}$ obey the condition $g \circ \Phi_{i,n,p}=\Phi_{i,n,p}$
for all $g \in \Sigma_p$ (where we use the above inclusion $\iota_{i,n,p}$).
We call such a datum a $T$-datum. $T$-data form a category equivalent to $T$-spectra.

We have to see that a $T$-datum is equivalent to a functor $\J \to \caC$.

Let be given a $T$-datum (in the notation as above).
Of course it is clear what the functor $\J \to \caC$ should do on objects, it should send
$(\i,\n)$ to $X_{i,n}$.

Let $(\varphi,\psi,\alpha) \colon (\i,\n) \to (\i \sqcup \p,\n \sqcup \p)$ be a map in $\J$.
We denote by $\Psi_{i,n,p} \colon (\i,\n) \to (\i \sqcup \p,\n \sqcup \p)$
the standard map, i.e. sending $\i$ to $\i$ via the identity map, similarly for $\n$, and
the required bijection is induced from the identity on $\p$. Then there is a
$(a,b) \in \Sigma_{i+p} \times \Sigma_{n+p}$ such that $(\varphi,\psi,\alpha)=
(a,b) \circ \Psi_{i,n,p}$. Moreover such a $(a,b)$ is unique with this property up to
precomposition with a $g \in \Sigma_p$ which acts via the inclusion $\iota_{i,n,p}$.

The image of the functor $\J \to \caC$ we want to define on the map
$(\varphi,\psi,\alpha)$ is defined to be $(a,b) \circ \Phi_{i,n,p}$. This is independent
of the possible choices for $(a,b)$ because of the last property of a $T$-datum.

This assignment indeed defines a functor $\J \to \caC$ because of the equivariance
of the maps in a $T$-datum.
One extends this to a functor from the category of $T$-data to the functor category
$\caC^\J$.

On the other hand starting with a functor $f \colon \J \to \caC$ defines a $T$-datum
by setting $X_{i,n}:= f((\i,\n))$ with the induced $\Sigma_i \times \Sigma_n$-action.
The maps $X_{i,n} \to X_{i+1,n+1}$ are defined to be the $f(\Psi_{i,n,1})$.
These are clearly $\Sigma_i \times \Sigma_n$-equivariant. Moreover the additional condition
on the $\Sigma_p$-invariance follows since in $\J$ we have the identity
$g \circ \Psi_{i,n,p}= \Psi_{i,n,p}$ for $g \in \Sigma_p$ acting via the inclusion $\iota_{i,n,p}$.

This assignment extends to a functor from $\caC^\J$ to $T$-data.

The two functors defined are clearly inverse to each other.

We have to see that these functors preserve the tensor product.

For that we describe the functor $\caC^\J \to \Spt_T^\Sigma(\caC^\Sigma)$ in different
terms. Let $j \colon \Sigma^2 \to \J$ be the embedding.
The functor $j$ induces a symmetric monoidal functor
$j_! \colon \caC^{\Sigma^2} \to \caC^\J$ with right adjoint
$j^*$. This adjunction induces an adjunction $\Mod(j^*(\unit_{\caC^\J})) \rightleftarrows \caC^\J$.
The left adjoint of this adjunction is given by the factorization
$$\Mod(j^*(\unit_{\caC^\J})) \to \Mod(j_!(j^*(\unit_{\caC^\J}))) \to \caC^\J,$$
where the last functor is given by push forward along the canonical map
of commutative monoids $j_!(j^*(\unit_{\caC^\J})) \to \unit_{\caC^\J}$.
Thus this left adjoint is also symmetric monoidal.

The unit $\unit_{\caC^\J}$ is given by $\Hom_\J ((\underline{0},\underline{0}), \_) \times \unit_\caC$.
Thus $\unit_{\caC^\J}(\m,\n)=\Sigma_n \times \unit_\caC$ if $m=n$ and the
initial object otherwise. It is easy to see that $j^*(\unit_{\caC^\J})$
is canonically isomorphic to $\mathrm{Sym}(T)$ as commutative monoids.
Moreover the canonical functor
$$\caC^\J \to \Mod(j^*(\unit_{\caC^\J})) \simeq \Mod(\mathrm{Sym}(T))
=\Spt_T^\Sigma(\caC^\Sigma)$$
is seen to be the functor desribed in the first part of this proof.
But we have already seen that its left adjoint is symmetric monoidal.
This finishes the proof.
\end{proof}

\begin{remark}
Suppose $\caC$ is a category with (set-indexed) coproducts. Then it is still
possible to desribe $\caC^\J$ in terms of symmetric spectra.
Namely, $(\caC^\Sigma)^\Sigma$ is tensored over $(\Set^\Sigma)^\Sigma$,
and $\caC^\J$ is equivalent to $\Mod_\caC(\mathrm{Sym}(T_\Set))$.
\end{remark}

Let now $\caC$ and $\caD$ be cocomplete symmetric monoidal categories
and $f \colon \caC \to \caD$ a symmetric monoidal functor.
Suppose $f$ is cocontinuous and the tensor product of $\caD$
commutes with colimits separately in each variable.
Let $K \in \caD$. Then there is a canonical cocontinuous symmetric monoidal functor
$f^K \colon \caC^\Sigma \to \caD$ which sends $T_\caC$ to $K$ and prolongs $f$. In formulas
it is given by
$$f^K((X_0,X_1,X_2,\ldots))= \coprod_{n \ge 0} f(X_n) \otimes_{\Sigma_n} K^{\otimes n}.$$
This functor extends to spectra
$$f^K_\Spt \colon \caC^\J \simeq \Spt_T^\Sigma(\caC^\Sigma) \to \Spt_K^\Sigma(\caD).$$

\begin{example} \label{hgfdsd}
Let $\caS$ be the category of spaces, i.e. either topological spaces
or simplicial sets and $\caS_\bullet$ the category of pointed spaces.
Let $f \colon \caS \to \caS_\bullet$ be the functor which adds a basepoint.
Then the induced functor $$f^{S^1}_\Spt \colon \caS^\J \to \Spt_{S^1}^\Sigma(\caS_\bullet)$$
is the functor $\mathbb{S}^\J[-]$ of \cite{sagave-schlichtkrull}.
\end{example}

\section{Model structures}
\label{gdsfgn}

We suppose now that $\caC$ and $\caD$ are left proper cellular or combinatorial symmetric
monoidal model categories and $f \colon \caC \to \caD$ is a symmetric monoidal left Quillen
functor. 
We suppose the unit in $\caC$ is cofibrant and $K \in \caD$ is cofibrant.
We equip the above mentioned categories of spectra
$\Spt_T^\Sigma(\caC^\Sigma)$ and $\Spt_K^\Sigma(\caD)$ with the stable model
stuctures of \cite{hovey.stable}. 
We have to verify that the required localizations exist in the combinatorial case.
Therefore we have to see that the categories $\Mod(\mathrm{Sym}(T))$ and/or
$\Mod(\mathrm{Sym}(K))$ are locally presentable.
This follows from \cite[Corollary 2.3.8.(1)]{lurie.DAGII}.

By transfer of structure the stable model structure on $\Spt_T^\Sigma(\caC^\Sigma)$
induces a model structure on the equivalent category $\caC^\J$. This is a localization
of the projective model structure on $\caC^\J$, the local objects are those diagrams for which
the transition maps are weak equivalences. It follows that in the case of spaces this model structure
is the same as the $\J$-model structure introduced by Sagave-Schlichtkrull \cite{sagave-schlichtkrull}.
The weak equivalences in this model structure are precisely the maps of diagrams
of spaces which induce weak equivalences on homotopy colimits, see loc.~cit.

\section{Graded endomorphism objects}

Let the notation be as in the last section.
Let $f^K_\Spt \colon \caC^\J \to \Spt_K^\Sigma(\caD)$ be the functor introduced in section \ref{hgrfddd}.
Wer claim this is a left Quillen functor. Indeed, for a discrete group $G$
and a $G$-object $X$ in $\caD$ which is underlying
cofibrant the functor $\caD[G] \to \caD$, $Y \mapsto Y \otimes_G X$,
is a left Quillen functor.

Let $\caO$ be an operad in $\caC$, e.g.~an $E_\infty$-operad (for example
$\caC$ could be a simplicial symmetric monoidal model category, then we can take the image
of an $E_\infty$-operad in simplicial sets). By abuse of notation we will also
talk about $\caO$-algebras in $\caD$ and the categories of symmetric spectra.
By this we shall mean algebras over the respective image of $\caO$.

Suppose $\E$ is an $\caO$-algebra in $\Spt^\Sigma_K(\caD)$. Let $r$ be the right adjoint
to $f^K_\Spt$. Then $r(\E)$ is also an $\caO$-algebra. If $\E$ is underlying fibrant
then $r(\E)$ has the correct homotopy type. 
In the case where $\caO$ is $\Sigma$-cofibrant one can always achieve this
by using semi model structures on $\caO$-algebras.

In particular cases one can think about $r(\E)$ as a graded endomorphism
object of $\E$, e.g.~for Example \ref{hgfdsd} and in the motivic situation,
see subsection \ref{hgrfdd}.

\subsection{Grading by tensor invertible objects}
\label{fghghn}

We suppose given cofibrant objects $K_1,\ldots,K_n \in \caD$ such that
these are tensor invertible in $\Ho \caD$. We let $\Spt^\Sigma_{\underline{K}}(\caD)$
be the model category of symmetric $K_1,\ldots,K_n$-multi-spectra in $\caD$.

By our assumption and \cite[Theorem 9.1]{hovey.stable}
the adjunction $\caD \rightleftarrows \Spt^\Sigma_{\underline{K}}(\caD)$
is a Quillen equivalence.

By iterating the definition of the functor $f^K_\Spt$
we get an induced symmetric monoidal left Quillen functor
$$f^{\underline{K}}_\Spt \colon \caC^{\J^n} \to \Spt^\Sigma_{\underline{K}}(\caD).$$

Suppose that $\caC$ is simplicial such that we have good models
of $E_\infty$-operads.

The derived right adjoint of $f^{\underline{K}}_\Spt$ on the level of $E_\infty$-algebras induces a functor
from the homotopy category of $E_\infty$-algebras in $\caD$ to the homotopy
category of $E_\infty$-algebras in $\caC^{\J^n}$, which can be thought of as multi-graded
$E_\infty$-algebras in $\caC$, see section \ref{gfththt}.

\subsection{The motivic example}
\label{hgrfdd}

We specialize the construction of subsection \ref{fghghn} to the motivic situation.
Let $\sS^\mot$ be the model category of $\mathbf{P}^1$-spectra for a given base scheme.
We let $K_1$ and $K_2$ be cofibrant models for the motivic spheres
$S^{1,0}$ and $S^{0,1}$. For the category $\caC$ we either take simplicial
sets with the natural functor $f$ to $\sS^\mot$ or the category of motivic spaces,
i.e simplicial presheaves on smooth schemes over the base scheme with an
$\mathbb{A}^1$- and Nisnevich-local model structure. The functor $f$ in this
case is the $\mathbf{P}^1$-suspension functor followed by adding a basepoint.

Let $r'$ be the right adjoint to the functor
$$f^{K_1,K_2}_\Spt \colon \caC^{\J^2} \to \Spt^\Sigma_{K_1,K_2}(\sS^\mot)$$
defined in subsection \ref{fghghn}.

Then the image with respect to $r'$
of a (fibrant) motivic $E_\infty$-ring spectrum $\E$ in $\caC^{\J^2}$ is a bigraded
version of the endomorphism space of $\E$:

\begin{definition}
Let $\E$ be a motivic $E_\infty$-spectrum and $\E \to R\E$ a fibrant replacement.
Then $r'(R\E)$ is defined to be the (derived) bigraded endomorphism space of $\E$.
\end{definition}

One can extract a bigraded version of $\mathrm{GL}_1 \E$ by taking the (sectionwise) invertible
endomorphisms (in the case $\caC$ is motivic spaces one has to work with a fibrant model).

\begin{remark}
One feature of $\J$-spaces is that they allow for a positive (flat) model structure
with the property that commutative $\J$-space monoids carry a model structure,
see \cite{sagave-schlichtkrull}. It is natural to expect that these features
carry over to the motivic setting, so that we can talk about strict commutative
algebra objects in $\caC^{\J^2}$ instead about $E_\infty$-algebras.
We note that positive flat model structures on motivic symmetric spectra and
algebras over arbitrary operads in motivic symmetric spectra have been worked
out in \cite{hornbostel.preorientations}.
\end{remark}

In the motivic stable homotopy category there are more tensor invertible
elements than just the motivic spheres, see \cite{hu.picard}. So one may
enlarge the number of grading directions for the graded units of
a motivic $E_\infty$-spectrum.

There might be relations among the tensor invertible elements.
This will not be reflected in our version of graded endomorphisms and units. We leave
this question to future work.

\section{The $\infty$-categorical interpretation}
\label{gfththt}

It is proven in \cite{sagave-schlichtkrull} that the classifying space
$\mathrm{B} \J$ is a model for $Q S^0$. Suppose for the rest that our model
categories are additionally simplicial model categories.

Our framework for $\infty$-categories will be mainly the weak Kan complexes resp.
quasi-categories, see e.g. \cite{lurie.HTT}.

We consider a usual category as an $\infty$-category by the nerve construction,
by abuse of notation we also write $C$ for the $\infty$-category associated
to the category $C$.

We view a topological space as an $\infty$-category via the singular simplicial set functor,
in particular classifying spaces of categories are viewed in such a way as $\infty$-categories.

Let $K$ be a simplicial set and $S \subset K_1$ a subset of the edges.
We denote by $K[S^{-1}]$ the pushout
$$\xymatrix{\coprod_{s \in S} \Delta^1 \ar[r] \ar[d] & K \ar[d] \\
\coprod_{s \in S} \overline{\Delta^1} \ar[r] & K[S^{-1}]},$$
where $\overline{\Delta^1}$ denotes the nerve of the category with
two objects and one isomorphism between these.

We note that the above pushout is a homotopy pushout both in the usual model
structure on $\sSet$ and the Joyal model structure.
Also $K \to K[S^{-1}]$ is a usual weak equivalence in $\sSet$.

We set $K[K^{-1}]=K[K_1^{-1}]$.

\begin{lemma} \label{hgfdsdg}
Let $C$ be an $\infty$-category, $K \in \sSet$ and $S \subset K_1$. Then
the map $C^{K[S^{-1}]} \to C^K$ is a fully faithful map between $\infty$-categories
whose essential image consists of those functors $K \to C$ which send each
edge in $S$ to an equivalence.
\end{lemma}
\begin{proof}
We first prove the special case saying that $C^{\overline{\Delta^1}} \to C^{\Delta^1}$
is fully faithful with essential image the edges in $C$ which are equivalences.
Let $s \colon \Delta^1 \to C$ be an equivalence. Then $s$ factors through
the maximal Kan complex in $C$. Thus the lifting property in the usual model structure
shows that $s$ can be extended to a map $\overline{s} \colon \overline{\Delta^1} \to C$.
This shows the claim about the essential image.

Let $f,g \colon \Delta^1 \to C$ be equivalences,
$f \colon x \to y$, $g \colon w \to z$.
Then the claim about the fully faithfulness
follows from the fact that there is a homotopy pullback 
diagram
$$\xymatrix{\Map_{C^{\Delta^1}}(f,g) \ar[r] \ar[d] & \Map_C(y,z) \ar[d] \\
\Map_C(x,w) \ar[r] & \Map_C(x,z)}$$
in $\Ho \sSet$ where every map is an isomorphism
and that $\overline{\Delta^1} \to \mathrm{pt}$ is a Joyal equivalence (so that the
mapping space in $C^{\overline{\Delta^1}}$ can be computed as a mapping space in $C$).

We prove now the general statement. Since the defining square for $K[S^{-1}]$ is a homotopy
pushout square the square

$$\xymatrix{C^{K[S^{-1}]} \ar[r] \ar[d] & \prod_{s \in S} C^{\overline{\Delta^1}} \ar[d] \\
C^K \ar[r] & \prod_{s \in S} C^{\Delta^1}}$$
is a homotopy pullback square in the Joyal model structure.
By what we have already proved the right vertical arrow
is fully faithful with essential image collections of arrows such that each arrow is a weak
equivalence. The claim follows from the fact that homotopy pullbacks of fully
faithful maps are fully faithful with essential image those objects which map to the essential
image of the given map.
\end{proof}

\begin{lemma} \label{iuyfgh}
Let $K \in \sSet$, $K[K^{-1}] \to R$ a Joyal fibrant replacement. Then $R$ is a Kan
complex. In particular the map $K \to R$ is a fibrant replacement in the usual model
structure on $\sSet$.
\end{lemma}
\begin{proof}
By \cite[Proposition 1.2.5.1]{lurie.HTT} an $\infty$-category $C$ is a Kan complex
if and only if the homotopy category $\mathrm{h}C$ associated to $C$
is a groupoid. So we have to show that $\mathrm{h} R$ is a groupoid.
But by construction $\mathrm{h} R$ is generated as a category by the morphisms
and their inverses which come from $K$. These are all invertible, thus
the claim follows.
\end{proof}

We keep the notation from section \ref{gdsfgn}.

\begin{proposition} \label{jhgddd}
The local model structure on $\caC^\J$ models
the $\infty$-category $\cC^{\mathrm{B} \J}\simeq \cC^{Q S^0}$, where $\cC$ is the $\infty$-category
associated to $\caC$. 
\end{proposition}
\begin{proof}
First observe that the projective model structure on $\caC^\J$ models the $\infty$-category
$\cC^\J$ by the strictification theorem \cite[Theorem 4.2.1]{toen-vezzosi.segal-topoi}.
One then checks that the local model structure on $\caC^\J$ models
the full subcategory $\cD$ of $\cC^\J$ consisting of functors which send all maps in $\J$ to equivalences.

Let $K$ be the nerve of $\J$. Then by Lemma \ref{hgfdsdg} the $\infty$-categories
$\cD$ and $\cC^{K[K^{-1}]}$ are canonically equivalent.

By Lemma \ref{iuyfgh} $K[K^{-1}]$ and $\mathrm{B} \J$ are canonically equivalent.
Thus we finally have
$\cD \simeq \cC^{K[K^{-1}]} \simeq \cC^{\mathrm{B} \J} \simeq \cC^{Q S^0}$.
\end{proof}

\begin{corollary}
The symmetric stabilization of the $\infty$-category $\cC^\Sigma$ with respect to the object
$(\emptyset,\unit,\emptyset,\emptyset, \ldots)$
(where we use the model category description for the symmetric stabilization)
is equivalent to the category $\cC^{Q S^0}$.
\end{corollary}
\begin{proof}
Tis follows from Propositions \ref{bgdedd} and \ref{jhgddd}.
\end{proof}

Pretending suitable $\infty$-categorical universal properties one may say that this
result states that $\cC^{Q S^0}$ is the universal symmetric monoidal $\infty$-category
with all colimits and a closed tensor product over $\cC$ generated by one
tensor invertible object. This universal property should hold in some $\infty$-category 
of such symmetric monoidal $\infty$-categories where the morphisms are the
symmetric monoidal cocontinuous functors.
Note that $\cC^\Sigma$ should be the universal cocomplete symmetric monoidal $\infty$-category
generated by one object. This free oject in $\cC^\Sigma$ is $(\emptyset,\unit,\emptyset,\emptyset, \ldots)$.

\vskip.2cm

Let be given a symmetric monoidal cocontinuous functor $F \colon \cC \to \cD$ between
cocomplete closed symmetric monoidal $\infty$-categories. Let $K \in \cD$
be tensor invertible. 
Suppose $\cC^{Q S^0}$ satisfies the universal property mentioned in the last paragraph.
Let $\cC^{Q S^0} \to \cD$ be the induced
symmetric monoidal functor sending the free tensor invertible object in $\cC^{Q S^0}$
to $K$. One may then ask if there exists
a factorization $\cC^{Q S^0} \to \cC^\integers \to \cD$. This should
express that the permutation actions on the $K^{\otimes n}$ are strictly
the identity, at least for the universal case where the functor $F$
is the identity.

Such a factorization for example exists for chain complexes
with $K=\integers[2]$ or motives with
$K=\integers(1)$ or $K=\integers(1)[2]$. This can be seen
by considering strong periodizations of the motivic
Eilenberg MacLane spectrum, see \cite{spitzweck.periodizable}.

There could also be intermediate factorizations, e.g.
through $\cC^{l_1(QS^0)}$. Here $l_1(QS^0)$ denotes the space obtained from $QS^0$
by killing all homotopy groups (of all connected components) above the first.

\bibliographystyle{plain}
\bibliography{J-spaces}

\vspace{0.1in}
\begin{center}
Department of Mathematics, University of Oslo, Norway.\\
e-mail: markussp@math.uio.no
\end{center}
\end{document}